\documentclass[11pt]{amsart}
\usepackage{calc,amssymb,amsmath,ulem}
\usepackage{alltt}
\RequirePackage[dvipsnames,usenames]{color}

\input{kmacros3.sty}
\input{xy}
\xyoption{all}
\usepackage{hyperref}
\usepackage{graphicx}
\usepackage[all,cmtip]{xy}
\usepackage{verbatim}
\usepackage[left=1.25in,top=1.09in,right=1.25in,bottom=1.09in]{geometry}

\theoremstyle{theorem}

\renewcommand{\O}{\mathcal O}
\newcommand{\ie}{{\itshape i.e.} }

\renewcommand{\to}[1][]{\xrightarrow{\ #1\ }}

\newcommand{\notdivide}{\not|\hspace{5pt}}

\begin{document}
\numberwithin{equation}{theorem}
\title[Bertini theorems for $F$-singularities]{Bertini theorems for $F$-singularities}
\author{Karl Schwede, Wenliang Zhang}
\begin{abstract}
We prove that strongly $F$-regular and $F$-pure singularities satisfy Bertini-type theorems (including in the context of pairs) by building upon a framework of Cumino, Greco and Manaresi (compare with the work of Jouanolou and Spreafico).  We also prove that $F$-injective singularities fail to satisfy even the most basic Bertini-type results.
\end{abstract}
\thanks{The first author was partially supported by the NSF grant DMS \#1064485.}
\thanks{The second author was partially supported by the NSF grant DMS \#1068946.}
\address{Department of Mathematics\\ The Pennsylvania State University\\ University Park, PA, 16802, USA}
\email{schwede@math.psu.edu}
\address{Department of Mathematics\\ University of Michigan\\ Ann Arbor, MI, 48109, USA}
\email{wlzhang@umich.edu}
\maketitle

\section{Introduction}

The study of projective varieties with ``mild'' (such as rational or log canonical) singularities is ubiquitous throughout higher dimensional complex algebraic geometry.  However, the classes of singularities that are so mild in characteristic zero, frequently exhibit much more pathological behavior in positive characteristic.  This is largely a consequence of the failure of Kodaira-type vanishing theorems.  However, there are classes of singularities with origins in commutative algebra (\ie tight closure theory) and representation theory (\ie Frobenius splitting theory) which seem better behaved in characteristic $p > 0$.  These are the so called \emph{$F$-singularities}.

In the study of varieties with mild singularities in characteristic zero, one of the more common tools is Bertini's theorem.  In particular, by cutting by general hyperplanes, many questions can be studied on smaller dimensional (presumably simpler) varieties.  Thus, the fact that ``mild'' singularities remain ``mild'' after cutting by general hyperplanes, is critical.
It is therefore natural to ask whether $F$-singularities are preserved when cut by general hyperplane sections.

In this paper we show that strongly $F$-regular and sharply $F$-pure singularities satisfy Bertini's second theorem (see Corollary \ref{cor.SecondTheoremOfBertini} for a precise statement).  In particular, they are preserved after taking general hyperplane sections.  We state our result in a special case first.
\vskip 6pt
\noindent{\bf Corollary \ref{cor.EasyBertiniStatement}.}  {\it Suppose that $X$ is a projective variety over an algebraically closed field $k$.  If $X$ is $F$-pure (respectively strongly $F$-regular) then so is a general hyperplane section of a very ample line bundle.}
\vskip 6pt
In particular, if $X$ is (globally) $F$-split, then a general hyperplane section is (locally) $F$-split.
We actually obtain a stronger statement which is written below.
\vskip 6pt
\noindent{\bf Theorem \ref{thm.MainTheorem}.}  {\it  Suppose that $X$ is a variety over an algebraically closed field $k$, let $\Delta \geq 0$ be a $\bQ$-divisor on $X$, and let $\phi:X\to \bP^n_k$ be a $k$-morphism with separably generated (not necessarily algebraic) residue field extensions. Suppose either
\begin{itemize}
\item[(i)]  $(X,\Delta)$ is strongly $F$-regular, or
\item[(ii)]  $(X,\Delta)$ is sharply $F$-pure, or
\item[(iii)]  $\Delta = 0$, $X$ is not necessarily normal, and $X$ is $F$-pure.
\end{itemize}
Then there exists a nonempty open subset $U$ of $(\bP^n_k)^{\vee}$, the dual projective space of hyperplanes, such that for each hyperplane $H\in U$,
\begin{itemize}
\item[(i)]  $\left(\phi^{-1}(H),\Delta|_{\phi^{-1}(H)}\right)$ is strongly $F$-regular.
\item[(ii)]  $\left(\phi^{-1}(H),\Delta|_{\phi^{-1}(H)}\right)$ is sharply $F$-pure.
\item[(iii)]  $\phi^{-1}(H)$ is $F$-pure.
\end{itemize}
}
\vskip 6pt
\noindent Parts (i) and (ii) of this result also extends to triples $(R, \Delta, \ba^t)$ by Remark \ref{rem.ExtensionToTriples}.

Strongly $F$-regular and sharply $F$-pure singularities are the moral equivalent of log terminal and log canonical singularities respectively, and so  Theorem \ref{thm.MainTheorem} should be viewed as an analog of \cite[Lemma 5.17]{KollarMori}.  The proof is an application of the axiomatic framework for proving Bertini-type theorems laid out in \cite{CuminoGrecoManaresiAxiomaticBertini}, \cf \cite{JouanolouTheoremDeBertini,SpreaficoAxiomaticTheoryForBertini}.  We also rely heavily on standard ideas for base change for $F$-singularities; \cf \cite{HochsterHunekeSmoothBaseChange}.

We mention that related questions have been studied before for strongly $F$-regular pairs in \cite[Example 4.7]{MustataYoshidaTestIdealVsMultiplierIdeals}.  In that work, \Mustata{} and Yoshida gave an example where general sections of a linear system behaved badly, and in particular, did not satisfy any version of Theorem \ref{thm.MainTheorem}.  However, their example had inseparable residue field extensions even though the extension of fraction fields was separably generated.  In particular, their example demonstrates that we cannot weaken the hypothesis of Theorem \ref{thm.MainTheorem} to only requiring separably generated residue field extensions only for the function field extension of the varieties.  See Remark \ref{rem.MusYos} for details.

Via $F$-inversion of adjunction applied to Theorem \ref{thm.MainTheorem}, we obtain the following corollary which should be compared with \cite[Lemma 5.17]{KollarMori}:
\vskip 6pt
\noindent{\bf Corollary \ref{cor.PairsCorollary}.}  {\it Suppose that $X$ is a variety over an algebraically closed field $k$, let $\Delta \geq 0$ be a $\bQ$-divisor on $X$.  Let $\phi:X\to \bP^n_k$ be a $k$-morphism with separably generated (not necessarily algebraic) residue field extensions.  Fix a general element $H$ of $(\bP^n_k)^{\vee}$.  Then:
\begin{itemize}
\item[(i)]  If $(X, \Delta)$ is sharply $F$-pure then $(X, \Delta + \phi^{-1}(H))$ is also sharply $F$-pure.
\item[(ii)]  If $(X, \Delta)$ is strongly $F$-regular, then $(X, \Delta+\phi^{-1}(H))$ is divisorially $F$-regular\footnote{The term ``\emph{divisorially} $F$-regular'' singularities unfortunately corresponds to ``\emph{purely} log terminal'' singularities \cite{TakagiPLTAdjoint}.} in the sense of \cite{HaraWatanabeFRegFPure}.
\item[(iii)]  If $(X, \Delta)$ is strongly $F$-regular, then $(X, \Delta+\varepsilon \phi^{-1}(H))$ is strongly $F$-regular for all $1 > \varepsilon \geq 0$.
\end{itemize}
}
\vskip 6pt
It is also natural to ask whether $F$-rational and $F$-injective singularities satisfy Bertini's second theorem.  We however show that $F$-injective singularities cannot satisfy Bertini's second theorem, see Section \ref{sec.WeakNormalityAndFailureOfBertini}.  We do not know how $F$-rationality behaves.
\vskip 6pt
\noindent{\bf Theorem \ref{thm.ExistFailureOfBertiniForSurfaces}.}  {\it
There exists a projective surface which is $F$-injective, except possibly at finitely many points, whose general hyperplane section is not $F$-injective.}
\vskip 6pt
\noindent
This example is based upon the study of weak normality and general hyperplane sections as developed in \cite{CuminoGrecoManaresiHyperplaneSectionsOfWNVarieties}.

\begin{remark}
We should point out that this paper is about global Bertini theorems.  A local Bertini theorem relates singularities at a point $x$ of a variety $X$ with singularities of general hypersurfaces through $x$.  Note that $F$-regular singularities do not satisfy local Bertini theorems, since for example, if $x \in X$ is an $F$-regular but non-regular point of a surface $X$ and $H$ is any hypersurface through $x \in X$, then $H$ is not even normal.  Indeed, if $H$ was normal, it would be regular and thus $x \in X$ would be regular as well.
\end{remark}
\vskip 9pt
\noindent{{\it{Acknowledgements:}}
\vskip 3pt
The authors would like to thank Florian Enescu for discussions on the history of base change problems for $F$-singularities and for comments on an earlier draft of this paper.  We are indebted to all of the referees, Alberto Fernandez Boix and Kazuma Shimomoto for numerous helpful comments on a previous draft of this paper. We would like to thank Mel Hochster for some very inspiring discussions, especially on properties described (A1) and (A1P) below.  We would like to thank S\'andor Kov\'acs for pointing out to us \cite{KollarHullsHusks}.  We would like to thank Shunsuke Takagi for some discussions on the problem in general.  Finally we would like to thank Kevin Tucker for useful discussions on base change problems and for comments on a previous draft of this paper.  The authors also thank Anton Geraschenko, Thomas Nevins, and {\tt{ulrich}} for providing some references to \cite{SpreaficoAxiomaticTheoryForBertini} and \cite{JouanolouTheoremDeBertini} on
\begin{center}
{\tt{http://mathoverflow.net/questions/67326/}}.
\end{center}

\section{Background and notation}

Throughout this paper, all schemes are separated and all rings and schemes are of characteristic $p > 0$ and $F$-finite.  It should be noted that $F$-finite rings are always excellent \cite{KunzOnNoetherianRingsOfCharP} and possess dualizing complexes \cite{Gabber.tStruc}.   Additionally, any scheme essentially of finite type over a perfect field is $F$-finite, and so since we are interested in schemes of finite type over algebraically closed fields, it is harmless to work in this setting.  We will pay special care that we do not depart from the $F$-finite setting in various base-change statements.  We also note that when we discuss a variety, it is always of finite type over an algebraically closed field.

Given a property $\sP$ of local rings, we say that a scheme $X$ satisfies $\sP$ if all of its stalks satisfy the property.  Such a property is called a \emph{local property of schemes}.

\begin{definition}
\label{def.GeometricallyP}
Suppose that $k$ is a field and that $X \to \Spec k$ is a map of schemes.  We say that $X$ is \emph{geometrically $\sP$} if for every finite extension $k' \supseteq k$, we have $X_{k'} := X \times_k k'$ is $\sP$.
\end{definition}

The reason we restrict to finite extensions is that we do not want to leave the category of $F$-finite schemes.   We could just as easily define geometrically $\sP$ by requiring that the $X_{k'}$ satisfy $\sP$ for finitely generated extensions $k' \supseteq k$ (this would be more general than finite extensions but less general than arbitrary extensions).  For example, note that $\bF_p(x_1, x_2, \dots, x_n)$ is $F$-finite, but $\bF_p(x_1, x_2, \dots)$ is not $F$-finite.

When dealing with a ring $R$ of characteristic $p > 0$, we use $F^e : R \to R$ to denote the $e$-iterated Frobenius map on $R$.  When $R$ is reduced (a setting in which we will always reside), we use $F^e_* R$ and $R^{1/p^e}$ interchangeably to denote the ring $R$ viewed as a module over itself via $F^e$ (\ie $r\cdot x = r^{p^e} x$).  The advantage of the latter notation is that it helps us distinguish elements from $R$ and $F^e_* R$.

Finally, we also recall the following generic freeness result which we will use several times.

\begin{theorem}[(Generic Freeness, Theorem 14.4 in \cite{EisenbudCommutativeAlgebra})]
\label{thm.GenFreeness}
Suppose that $S$ is a Noetherian domain and $T$ is a finitely generated $S$-algebra. If $M$ is a finitely generated $T$-module, then there exists a nonzero element $c\in S$ such that $M[c^{-1}]$ is a free $S[c^{-1}]$-module.
\end{theorem}

\subsection{Definitions of $F$-singularities}

We recall various definitions of $F$-singularities.

\begin{definition}[(strong $F$-regularity)]
Suppose that $(R,\bm)$ is an $F$-finite reduced local ring of characteristic $p > 0$. We say that $R$ is \emph{strongly $F$-regular} if, for each $c\in R$ not in any minimal prime of $R$, there exists $q=p^e$ such that the $R$-linear map $R\xrightarrow{1\mapsto c^{1/q}}R^{1/q}$ splits, or equivalently,  for each $c\in R$ not in any minimal prime of $R$, there exists $q=p^e$ such that the $R$-linear map $(R\xrightarrow{1\mapsto c^{1/q}}R^{1/q})\otimes_RE$ is injective, where $E$ is the injective hull of the residue field of $R$.
\end{definition}

\begin{definition}[($F$-purity)]
Suppose that $(R,\bm)$ is an $F$-finite reduced local ring of characteristic $p > 0$. We say that $R$ is \emph{$F$-pure} if $R\hookrightarrow R^{1/p}$ splits, or equivalently, $(R\hookrightarrow R^{1/p})\otimes_RE$ is injective, where $E$ is the injective hull of the residue field of $R$.
\end{definition}

\begin{definition}[($F$-injectivity)]
Suppose that $(R,\bm)$ is an $F$-finite local ring of characteristic $p > 0$.  We say that $R$ is \emph{$F$-injective} if the Frobenius map $F : H^i_{\bm}(R) \to H^i_{\bm}(F^e_* R)$ is injective for every $i \geq 0$.  Dually, $R$ is $F$-injective if for each $i \in \bZ$, the $i$th cohomology of the Grothendieck-trace map
\[
\myH^i(F^e_* \omega_R^{\mydot}) \to \myH^i \omega_R^{\mydot}
\]
is surjective for all $e > 0$ (or equivalently some $e > 0$).  Here $\omega_R^{\mydot}$ is a dualizing complex for $R$.
\end{definition}

\begin{remark}
It is worth remarking that $F$-injective rings are automatically reduced.  Indeed, suppose that $0 \neq f \in R$ is a nilpotent element.  Set $J = \Ann_R(f)$ and suppose that $Q$ is a minimal prime containing $J$.  First consider $f/1 \in R_Q$.  If $f/1 = 0$, then $f$ is killed by an element of $R \setminus Q \subseteq R \setminus \Ann_R(f)$, which is impossible.  Thus we know that $0 \neq f/1 \in R_Q$.  Furthermore, we notice that $\Ann_{R_Q}(f/1)$ is $QR_Q$-primary.  In particular, $Q^n R_Q \subseteq \Ann_{R_Q}(f/1)$.  But then $0 \neq f/1 \in H^0_{QR_Q}(R_Q) \subseteq R_Q$.  Now, some iterate of Frobenius kills $f/1$ since $f$ was nilpotent.  Therefore the Frobenius map $F : H^0_{QR_Q}(R_Q) \to H^0_{QR_Q}(F_* R_Q)$ is not injective.  By local duality, the map $\myH^0(F_* \omega_{R_Q}^{\mydot}) \to \myH^0(\omega_{R_Q}^{\mydot})$ is not surjective.  It follows immediately that $R_Q$ is not $F$-injective and thus $R$ is also not $F$-injective.
\end{remark}

\begin{definition}[($F$-rationality)]
We say that $(R, \bm)$ is \emph{$F$-rational} if it is Cohen-Macaulay and there is \emph{no} non-zero submodule $J \subsetneq \omega_R$ such that the Grothendieck-trace map $\Phi : F^e_* \omega_R \to \omega_R$ satisfies $\Phi(F^e_* J) \subseteq J$.  In other words, if $\omega_R$ is simple under the action of $\Phi$ and $R$.
\end{definition}

All of these notions extend to not-necessarily-local rings (and schemes) by requiring the condition at every point.  In the $F$-finite case, they are all known to be open conditions.  We now mention a generalization of $F$-purity and strong $F$-regularity to pairs.  We will not review basic facts about the formalism of $\bQ$-divisors in this paper.  Instead we suggest the reader see \cite{KollarMori}, \cite{HaraWatanabeFRegFPure} or \cite{SchwedeTuckerTestIdealSurvey} for a discussion of $\bQ$-divisors in this context.

\begin{definition}
A \emph{pair}, denoted $(R, \Delta)$ (or $(X, \Delta)$) is the combined information of a normal ring $R$ (respectively, a normal scheme $X$) and an effective $\bQ$-divisor $\Delta \geq 0$.
\end{definition}

\begin{definition}[(strong $F$-regularity for pairs)]
A pair $(R, \Delta)$, with $R$ local, is said to be \emph{strongly $F$-regular} if for each nonzero $c \in R$ , the composition
\[
\begin{array}{rcccl}
R & \to & F^e_* R & \to & F^e_* (R(\lceil (p^e - 1)\Delta \rceil))\\
1 & \mapsto & c & \mapsto & c
\end{array}
\]
splits for some\footnote{equivalently, all $e \gg 0$ or infinitely many $e > 0$} $e > 0$ .
\end{definition}

\begin{remark}
For strong $F$-regularity, it is sufficient to obtain a splitting for a single $c$ such that $\Supp(\Delta) \subseteq V(c)$ and $\Spec R[c^{-1}]$ is regular.  See Lemma \ref{lem.FPurePlusEpsilonImpliesFReg} below for a slight variation on this statement.
\end{remark}

\begin{definition}[(sharp $F$-purity for pairs)]
A pair $(R, \Delta)$ is said to be \emph{sharply $F$-pure} if the composition
\[
\begin{array}{rcccl}
R & \to & F^e_* R & \to & F^e_* (R(\lceil (p^e - 1)\Delta \rceil))\\
1 & \mapsto & 1 & \mapsto & 1
\end{array}
\]
splits for some\footnote{equivalently, infinitely many $e > 0$ or all sufficiently divisible $e > 0$} $e > 0$.
\end{definition}

\begin{remark}
\label{rem.SharpFPurePairsAreNormal}
Notice that a sharply $F$-pure (respectively strongly $F$-regular) pair $(R, \Delta \neq 0)$, $R$ is by definition normal.  In particular, geometrically sharply $F$-pure pairs are geometrically normal as well.
\end{remark}

We also recall a common characterization of (sharply) $F$-pure pairs essentially taken from \cite{HaraWatanabeFRegFPure}.

\begin{lemma} [(Proposition 2.4(1) in \cite{HaraWatanabeFRegFPure})]
\label{lem.MatlisDualSharplyFPure}
Suppose that $(R, \bm)$ is an $F$-finite normal $d$-dimensional local ring and $\Delta$ is a $\bQ$-divisor on $X = \Spec R$.  Then $(R, \Delta)$ is sharply $F$-pure, if and only if there exists an $e > 0$ such that the composition
\[
\begin{array}{rcl}
E_R & \cong & H^d_{\bm}( \omega_R ) \\
& \cong & H^d_{\bm}( \O_X(K_X)) \\
& \to & H^d_{\bm}(F^e_* \O_X(p^e K_X)) \\
& \to & H^d_{\bm}(F^e_* \O_X(p^e K_X + \lceil (p^e - 1)\Delta \rceil)) \\
& \cong & H^d_{\bm}(F^e_* \O_X(K_X + \lceil (p^e - 1)(K_X + \Delta) \rceil))
\end{array}
\]
is injective.
\end{lemma}
\begin{proof}
This is essentially contained in \cite[Proposition 2.4(1)]{HaraWatanabeFRegFPure} but we give a brief proof here for the convenience of the reader.
By Matlis/local-duality \cite{HartshorneResidues}, the injectivity above is equivalent to the surjectivity of a composition:
\[
F^e_* \O_X( -\lceil (p^e - 1)(K_X + \Delta) \rceil) \to F^e_* \O_X( (1 - p^e)K_X) \to \O_X
\]
If this composition is surjective, we may choose an element $z \in F^e_* \O_X( -\lceil (p^e - 1)(K_X + \Delta) \rceil)$ sent to $1 \in \O_X$ and notice that the submodule generated by $z$ is isomorphic to $F^e_* \O_X$.  This is our splitting, the factorization above guarantees that our splitting is of the desired form.

For the converse direction, notice that
\[
\O_X(K_X) \to F^e_* \O_X(p^e K_X)  \to F^e_* (\O_X(p^e K_X + \lceil (p^e - 1)\Delta \rceil))
\]
splits and then apply $H^d_{\bm}(\bullet)$.
\end{proof}

\begin{remark}
We now enumerate some basic properties of pairs:
\label{rem.BasicPropertiesOfFPurePairs}
\begin{enumerate}
\item One can see that the pair $(R,\Delta)$ is sharply $F$-pure if and only if the evaluation (at 1) map
\[\Hom_R(F^e_* (R(\lceil (p^e - 1)\Delta \rceil),R) \to R\]
is surjective for some $e > 0$. And since surjectivity is a local property, one can check sharp $F$-purity locally at maximal ideals of $R$.
\item It follows from (i) that if $(R_{\bp}, \Delta|_{\Spec(R_{\bp})})$ is strongly $F$-regular (or sharply $F$-pure) for some $\bp\in \Spec(R)$, then there is an open neighborhood $U$ of $\bp$ in $\Spec(R)$ such that $(R_{\bp'}, \Delta|_{\Spec(R_{\bp'})}))$ is strongly $F$-regular (or sharply $F$-pure) for each $\bp'\in U$.
\item If $(p^e - 1)\Delta$ is an integral Weil divisor and $(R, \Delta)$ is sharply $F$-pure, then it is easy to see that the map $R \to F^e_* (R((p^e - 1)\Delta))$ splits.  This holds because this map can be used to factor a map that does indeed split (just as if $R \to R^{1/p^e}$ splits then so does $R \to R^{1/p}$).
\end{enumerate}
\end{remark}

We will use the following characterization of strong $F$-regularity via sharp $F$-purity.

\begin{lemma} [(Lemma 5.9(a) in \cite{HochsterHunekeSmoothBaseChange}, Corollary 3.10 in \cite{SchwedeSmithLogFanoVsGloballyFRegular})]
\label{lem.FPurePlusEpsilonImpliesFReg}
Suppose that $R$ and $\Delta$ are as above and $\Gamma > 0$ is any other $\bQ$-divisor whose support contains the locus where $(R, \Delta)$ is not strongly $F$-regular.  If $(R, \Delta + \Gamma)$ is sharply $F$-pure, then $(R, \Delta)$ is strongly $F$-regular.
\end{lemma}

We also remind the reader that in many cases, the divisor $\Delta$ in sharply $F$-pure pairs can be assumed to be in a particularly nice form.

\begin{lemma} [(Proposition 3.12 in \cite{SchwedeSmithLogFanoVsGloballyFRegular})]
\label{lem.SharpFPureCanBeAssumedNice}
Suppose that $(R, \Delta)$ is sharply $F$-pure.  Then there exists a divisor $\Delta ' \geq \Delta$ such that $(R, \Delta')$ is sharply $F$-pure and such that $K_R + \Delta'$ is $\bQ$-Cartier with index not divisible by $p > 0$.  Additionally, one can arrange $e > 0$ such that $(p^e - 1) \Delta'$ is integral and that $\Hom_R(F^e_* R((p^e - 1)\Delta'), R) \cong F^e_* R( (1-p^e)(K_R + \Delta'))$ is isomorphic to $F^e_* R$ as an $F^e_* R$-module.
\end{lemma}

We also will need a weak global generalization of this statement stated below.

\begin{lemma}
\label{lem.GlobalSharpFPureCanBeAssumedNice}
Suppose that $(X, \Delta)$ is sharply $F$-pure.  Then there exists a $\Delta' \geq \Delta$ such that $(X, \Delta')$ is sharply $F$-pure and that $(p^e - 1)\Delta'$ is an integral Weil divisor for some $e > 0$.
\end{lemma}
\begin{proof}
The fact that $(X, \Delta)$ is sharply $F$-pure is equivalent to requiring that the evaluation-at-1 map:
\[
\sHom_{\O_X}(F^e_* \O_X(\lceil (p^e - 1) \Delta\rceil), \O_X) \to \O_X
\]
is a surjective map of sheaves for some $e > 0$ (if it holds at a point, it holds in a neighborhood for all sufficiently divisible $e$).  Set $\Delta' = {1 \over p^e - 1}\lceil (p^e - 1) \Delta \rceil$.  The result follows immediately.
\end{proof}

Finally, we recall some well known properties of $F$-singularities of pairs.

\begin{lemma} [(Corollary 3.10, Lemma 3.5, Theorem 3.9 in \cite{SchwedeSmithLogFanoVsGloballyFRegular})]
\label{lem.BasicPropertiesOfFpurityAndRegularity}
Suppose that $(R, \Delta)$ is a strongly $F$-regular pair.
\begin{itemize}
\item[(a)]  If $\Gamma > 0$ is any other $\bQ$-divisor, then for all $1 \gg \varepsilon > 0$, we have that $(R, \Delta + \varepsilon \Gamma)$ is also strongly $F$-regular.
\item[(b)]  For all $1 > s \geq 0$, we have that $(R, s\Delta)$ is strongly $F$-regular\footnote{The analogous statement also holds for sharp $F$-purity}.
\item[(c)]  If  $(R, \Delta + \Gamma)$ is sharply $F$-pure and $(R, \Delta)$ is strongly $F$-regular, then $(R, \Delta + \varepsilon \Gamma)$ is strongly $F$-regular for all $1 > \varepsilon \geq 0$.
\end{itemize}
\end{lemma}

\subsection{Pulling back Weil divisors under flat maps}

We will repeatedly pull back Weil divisors under flat maps in this paper and so it is important that we briefly discuss this process.

Suppose that $R \to S$ is a flat map of normal domains corresponding to $f : \Spec S \to \Spec R$.  Suppose that $M$ is any finitely generated reflexive $R$-module (\ie the natural map is an isomorphism $M \cong \Hom_R(\Hom_R(M, R), R)$), then $M \tensor_R S$ is also $S$-reflexive.  This is easy since $M \tensor_R S \cong \Hom_R(\Hom_R(M, R), R) \tensor_R S \cong \Hom_S(\Hom_S(M \tensor_R S, S), S)$ because $S$ is $R$-flat.

Recall that an effective integral Weil divisor on $\Spec R$ is simply a choice of a rank-1 reflexive $R$-module $M$ with a choice of a map $R \to M$ (up to multiplication by a unit of $R$).  Tensoring such a module by $S$ gives us a map $S \to M \tensor_R S$ and so determines an effective Weil divisor on $\Spec S$, in particular we have a map $f^*$ from $\Div(\Spec R)$ to $\Div(\Spec S)$.  If $D$ and $E$ are effective divisors corresponding to reflexive sheaves $M$ and $N$ respectively, then $D+E$ corresponds to $(M \tensor N)^{\vee\vee}$ (where $\bullet^{\vee}$ denotes $\Hom_R(\bullet, R)$).  It follows that $f^*(D+E)$ corresponds to $f^*D + f^*E$.  The same formulas hold for anti-effective divisors which correspond to inclusions $M \subseteq R$.  Finally, a similar set of arguments yield that $f^*(-D) = -f^*(D)$ and so in general, we have a group homomorphism: $f^* :  \Div(\Spec R) \to \Div(\Spec S)$.

We can extend this to $\bQ$-divisors formally.  However, we should point out that if $D$ is a $\bQ$-divisor on $\Spec R$, then $\lceil f^* D \rceil \leq f^* \lceil D \rceil$.  To see this, write $\lceil D \rceil = D + B$ where $B$ is an effective $\bQ$-divisor.  It follows that $f^* \lceil D \rceil = f^* D + f^* B$ and since $f^* B$ is still effective, the claim follows.

\subsection{Divisors on fibers}
Finally, suppose that $Y \to X$ is a map of finite type where $X$ is a variety over an algebraically closed field, $Y$ is geometrically normal over $X$ (meaning that every fiber is geometrically normal in the sense of Definition \ref{def.GeometricallyP}), and $\Delta$ is a $\bQ$-divisor on $Y$.  Then for an open dense set of points $s \in X$, $Y_s$ is normal and we obtain divisors $\Delta_s = \Delta|_{Y_s}$ on the fibers $Y_s$.  To see this claim, it is harmless to assume that $X$ is regular and $\Delta$ is an integral prime Cartier divisor (by removing the singular locus of $Y$ and working by linearity).  But then the claim follows directly from generic freeness \cite[Theorem 14.4]{EisenbudCommutativeAlgebra} (assume that $\O_{\Delta}$ is free over $\O_X$).  We now also present an alternate way to restrict Weil divisors that we will certainly use.

\begin{lemma}
With notation as above, if $M$ is a rank-1 reflexive module on $Y$, then for an open dense set of fibers $Y_s$, $M_s = M \tensor_{\O_X} k(s)$ is a rank-1 reflexive module on $Y_s$.
\end{lemma}
It follows immediately that if $M$ corresponds to a Weil divisor $D$ then $M_s$ corresponds to $D|_{Y_s} =: D_s$.  This lemma also follows from \cite[Theorem 2]{KollarHullsHusks}.

\begin{proof}
Without loss of generality we may assume that $X = \Spec B$ is regular and that $Y = \Spec R$ is normal.  Fix $N$ such that $M = \Hom_R(N, R)$.  By generic freeness, replacing $X = \Spec B$ with an open subset, we may assume that $M$, $N$, $R$, as well as $\Ext^0_R(N, R), \dots, \Ext^{\dim X}_R(N, R)$ are all $B$-free.  Fix a point $s \in X$.  By replacing $B$ by $B_s$, which is $B$ localized at $s$ (which is flat) and $R$ by $R \tensor_B B_s$, we may assume that $s$ is generated by a regular sequence $x_1, \dots, x_n \in B$.

We know that
\begin{equation}
\label{eq.DerivedHomTensorFTD}
\myR \Hom^{\mydot}_R(N, R) \tensor^{\bf L}_R R/\langle x_1, \dots, x_n \rangle \cong \myR \Hom^{\mydot}_R(N, R/\langle x_1, \dots, x_n \rangle)
\end{equation}
by \cite[Chapter II, 5.14]{HartshorneResidues}.  Taking zeroth cohomology of the left side is computed by $\Tor_i^R(\Ext^i_R(N, R), R/\langle x_1, \dots, x_n\rangle)$ for $i = 0, \dots, n$.  But these are zero for $i > 0$ since $\Ext^i_R(N, R)$ is $B$-free for $i = 1, \dots, n$ (and hence $x_1, \dots, x_n$ form a regular sequence on it).  Thus
\[
\myH^0(\myR \Hom^{\mydot}_R(N, R) \tensor^{\bf L}_R R/\langle x_1, \dots, x_n \rangle) \cong \Hom_R(N, R) \tensor R/\langle x_1, \dots, x_n \rangle.
\]
Taking zeroth cohomology of Equation \ref{eq.DerivedHomTensorFTD} then yields:
\[
M \tensor_R R/\langle x_1, \dots, x_n \rangle \cong \Hom_R(N, R) \tensor_R R/\langle x_1, \dots, x_n \rangle \cong \Hom_R(N, R/\langle x_1, \dots, x_n \rangle).
\]
But
\[
\Hom_R(N, R/\langle x_1, \dots, x_n \rangle) \cong \Hom_{R/\langle x_1, \dots, x_n \rangle}(N/\langle x_1, \dots, x_n \rangle, R/\langle x_1, \dots, x_n \rangle).
\]
It then follows that $M|_{Y_s} = M \tensor_R R/\langle x_1, \dots, x_n \rangle$ is reflexive which is what we wanted to prove.
\end{proof}

\section{Restating Cumino-Greco-Manaresi for pairs}

In \cite{CuminoGrecoManaresiAxiomaticBertini}, the authors developed a framework for proving Bertini-theorems for classes of singularities in characteristic $p > 0$.  The idea of the proof is the same as in related results contained in \cite{JouanolouTheoremDeBertini} and similar ideas have appeared in a number of other expositions as well, also see \cite{ZariskiTheoremOfBertiniOnTheVariableSingularPoints,SeidenbergHyperplaneSectionOfNormal}.  For an excellent historical survey of the origins of Bertini's theorem see \cite{KleimanBertiniHistory}.  The framework of Cumino-Greco-Manaresi is convenient for our purpose however.
We note that these ideas were further developed in \cite{SpreaficoAxiomaticTheoryForBertini} in the case of non-algebraically closed fields (and in other directions).  However, we work in the algebraically closed case for simplicity, although several of the results of this paper do hold more generally with no additional work \cf \cite[Corollaries 4.3, 4.5]{SpreaficoAxiomaticTheoryForBertini}.

We recall the following axioms for a local property $\scr{P}$ in case of locally Noetherian schemes.  These are taken directly from Cumino-Greco-Manaresi \cite{CuminoGrecoManaresiAxiomaticBertini}, \cf \cite{SpreaficoAxiomaticTheoryForBertini}:
\begin{enumerate}
\item[(A1)] whenever $\phi:Y\to Z$ is a flat morphism with regular fibers and $Z$ is $\scr{P}$, then $Y$ is $\scr{P}$ too;
\item[(A2)] let $\phi:Y\to S$ be a morphism of finite type, where $Y$ is excellent and $S$ is integral with generic point\footnote{The point corresponding to the zero ideal locally.} $\eta$; if $Y_{\eta}$ is geometrically $\scr{P}$, then there exists an open neighborhood $U$ of $\eta$ in $S$ such that $Y_s$ is geometrically $\scr{P}$ for each $s\in U$.
\item[(A3)] $\scr{P}$ is open on schemes of finite type over a field.
\end{enumerate}

\begin{remark}
For a partial list of properties satisfying (A2) (with citations to EGA), see \cite[Appendix E]{GortzWedhornAlgebraicGeometry1}.
\end{remark}

Cumino, Greco, and Manaresi proved in \cite{CuminoGrecoManaresiAxiomaticBertini} that, once a property $\scr{P}$ on an algebraic variety $V$ over an algebraically closed field satisfies the three axioms (A1), (A2) and (A3), then the Second Theorem of Bertini holds for $\scr{P}$. More precisely,

\begin{theorem}[(Theorem 1 in \cite{CuminoGrecoManaresiAxiomaticBertini})]
\label{thm.1FromCGM}
Let $X$ be a scheme of finite type over an algebraically closed field $k$, let $\phi:X\to \bP^n_k$ be a morphism with separably generated (not necessarily algebraic) residue field extensions. Suppose $X$ has a local property $\scr{P}$ verifying $({\rm A1})$ and $({\rm A2})$. Then there exists a nonempty open subset $U$ of $(\bP^n_k)^{\vee}$ such that $\phi^{-1}(H)$ has the property $\scr{P}$ for each hyperplane $H\in U$.
\end{theorem}

\begin{remark}
In the above statement, and throughout the paper, we use $(\bP^n_k)^{\vee}$ to denote the dual-projective space of hyperplanes of $\bP^n_k$.
\end{remark}

The following corollaries follow easily.
\begin{corollary}[(Second Theorem of Bertini, Corollary 1 in \cite{CuminoGrecoManaresiAxiomaticBertini})]
\label{cor.SecondTheoremOfBertini}
Let $V$ be an algebraic variety over $k=\bar{k}$, and let $S$ be a finite dimensional linear system on $V$. Assume that the rational map $V\dashrightarrow \mathbb{P}^n$ corresponding to $S$ induces (whenever defined) separably generated field extensions. Let $\scr{P}$ be a property satisfying ${\rm (A1)}$, ${\rm (A2)}$, and ${\rm (A3)}$. Then the general element of $S$, considered as a subscheme of $V$, has the property $\scr{P}$ but perhaps at the base points of $S$ and at the points of $V$ which are not $\scr{P}$.
\end{corollary}

Also, whenever $\scr{P}$ satisfies the three axioms (A1), (A2) and (A3), then $\scr{P}$ is preserved by the general hyperplane section.

\begin{corollary}[(Corollary 2 in \cite{CuminoGrecoManaresiAxiomaticBertini})]
Let $V\subseteq \mathbb{P}^n$ be a closed subscheme and let $\scr{P}$ be a local property satisfying ${\rm (A1)}$ and ${\rm (A2)}$. Then if $V$ is $\scr{P}$, the general hyperplane section of $V$ is $\scr{P}$. If moreover $\scr{P}$ satisfies ${\rm (A3)}$, then the $\scr{P}$-locus of $V$ is preserved by the general hyperplane section; that is for the general hyperplane section $H$ of $\mathbb{P}^n$ one has $\scr{P}(V\cap H)\supseteq \scr{P}(V)\cap H$.
\end{corollary}

This result even holds if $V$ is contained in a hyperplane $H$ of $\mathbb{P}^n$, since the linear system of hyperplanes on $\mathbb{P}^n$ restricts to the complete linear system of hyperplanes on $H$.

In this paper, we will prove the Second Theorem of Bertini for classes of $F$-singularities for pairs. To this end, we have to modify Cumino-Greco-Manaresi's axioms for pairs as follows:

\begin{enumerate}
\item[(A1P)] Let $\phi:Y\to Z$ be a flat morphism with regular fibers and $\Delta$ be a $\bQ$-divisor on $Z$. If $(Z,\Delta)$ is $\scr{P}$, then $(Y,\phi^*\Delta)$ is $\scr{P}$ too.
\item[(A2P)]Let $\kappa: Y \to Z$ be a morphism of finite type of $F$-finite schemes, where $Z$ is an integral scheme with generic point $\eta$.  Fix a $\bQ$-divisor\footnote{If $\Delta \neq 0$ we assume $Y_{\eta}$ is normal, otherwise we make no such assumption.} $\Delta$ on $Y$ such that $(Y_{\eta}, \Delta_{\eta})$ is geometrically $\sP$.  Then there
exists an open neighborhood $U$ of $\eta$ in $Z$ such that $(Y_s, \Delta_s)$ is geometrically $\sP$ for each $s \in U$.
\item[(A3P)]  $\sP$ is an open condition for pairs $(Y, \Delta)$ of finite type over an $F$-finite field.
\end{enumerate}

\begin{remark}
Note that built into (A2P) is the requirement that $(Y_s, \Delta_s)$ makes sense.  In particular, we require that $\Delta_s$ can be interpreted as a $\bQ$-divisor on $Y_s$ (of course this holds for an open dense set of points on $Z$ since $\kappa$ is generically flat.)
\end{remark}

Following the strategy in \cite{CuminoGrecoManaresiAxiomaticBertini}, we need only to establish
\begin{theorem}
\label{thm.CGMForPairs}
Let $X$ be a normal scheme of finite type over an algebraically closed field $k$, let $\Delta$ be a $\bQ$-divisor on $X$, and let $\phi:X\to \bP^n_k$ be a morphism with separably generated (not necessarily algebraic) residue field extensions. Suppose $(X,\Delta)$ has a local property $\scr{P}$ verifying $({\rm A1P})$ and $({\rm A2P})$. Then there exists a nonempty open subset $U$ of $(\bP^n_k)^{\vee}$ such that $\left(\phi^{-1}(H),\Delta|_{\phi^{-1}(H)}\right)$ has the property $\scr{P}$ for each hyperplane $H\in U$.
\end{theorem}
\begin{proof}
The following proof is taken directly from \cite{CuminoGrecoManaresiAxiomaticBertini}.  All we add is the divisor pair.

Let $Z$ be the reduced closed subscheme of $\bP^n_k \times_k (\bP^n_k)^{\vee}$ obtained by taking the closure of the set
\[
\big\{(x,H) \in \,\big|\, x \in H \big\}.
\]
We claim that the projection map $\beta : Z \to \bP^n_k$ is flat.  Indeed, it is clearly generically flat, and since it is finite type it is flat at some closed point.  But for any $z \in Z$, $\O_{Z, z}$ is certainly isomorphic to that of any other $z$.  Furthermore, the inclusion of rings $\O_{\bP^n_k, \beta(z)} \to \O_{Z, z}$ is also independent of the choice of $z$, up to isomorphism.  Thus $\beta$ is flat in general.

Just as in \cite{CuminoGrecoManaresiAxiomaticBertini}, we form a commutative diagram:
\[
\xymatrix{
Y := X \times_{\bP^n_k} Z \ar[dd]_{\gamma} \ar[dr]_{\rho} \ar[rr]^{\sigma} & & \ar[dl]_{\pi} Z \ar@{^{(}->}[d] \\
& (\bP^n_k)^{\vee} & \bP^n_k \times_k (\bP^n_k)^{\vee} \ar[d] \ar[l]^-{\pi'} \\
X \ar[rr]_{\phi} & & \bP^n_k
}
\]
We describe each map appearing above.
\begin{itemize}
\item{} $\sigma$ is the projection.
\item{} $\pi'$ is the projection and thus so is $\pi$.
\item{} $\rho = \pi \circ \sigma$.
\item{} $\gamma$ is the projection.  Note that $\gamma$ is flat since it is a base change of the projection $Z \to \bP^n_k$.
\end{itemize}
For each hyperplane $H \subseteq \bP^n_k$ viewed also as a point in $(\bP^n_k)^{\vee}$, we have
\[
\phi^{-1}(H) \cong \sigma^{-1}(\pi^{-1}(H)) = \rho^{-1}(H).
\]
We obtain a divisor $\Delta_Y$ on $Y := X \times_{\bP^n_k} Z$.  Therefore we want to show that for a dense open set of points of $(\bP^n_k)^{\vee}$, that the fibers of $\rho$, which yield pairs $(Y_{s}, (\Delta_{Y})_s)$, satisfy property $\sP$.  Set $\eta$ to be the generic point of $(\bP^n_k)^{\vee}$.  By (A2P), we merely need to show that the generic fiber\footnote{in other words, the fiber of the generic point} $(Y_{\eta}, (\Delta_{Y})_{\eta})$ is geometrically $\sP$.  Therefore, we only need to show that for each finite field extension $K$ over $k(\eta)$, that $(Y_{K}, (\Delta_{Y})_{K})$ is $\sP$.

Consider the following composition which we denote by $q$:
\[
Y_K = Y_{\eta} \times_{k(\eta)} K = (X \times_{\bP^n_k} Z_{\eta}) \times_{k(\eta)} K \to X \times_{\bP^n_k} Z_{\eta} \to X \times_{\bP^n_k} Z \xrightarrow{\gamma} X
\]
where the final map is the projection.  Notice that this map is flat since each map in the composition is flat.  In \cite{CuminoGrecoManaresiAxiomaticBertini}, Cumino, Greco and Manaresi proved that $q$ has regular fibers.  Thus by (A1P), the pair $(Y_{K}, q^* \Delta) = (Y_{K}, (\Delta_{Y})_K)$ is $\sP$.  This completes the proof.
\end{proof}

\section{\textnormal{(A1P)} for sharply $F$-pure and strongly $F$-regular pairs}

The property (A1) has been heavily studied for $F$-singularities.  In this section, we generalize these results to pairs.  First however, we review some history.  Such problems were first studied in \cite[Section 7]{HochsterHunekeSmoothBaseChange} where it was shown that $F$-regularity is well behaved in many cases (and the same arguments imply that $F$-purity is equally well behaved).  In \cite{VelezOpennessOfTheFRationalLocus}, V\'elez showed that $F$-rationality behaved well for smooth morphisms.  In \cite{EnescuBehaviorOfFrationalUnderBaseChange,HashimotoCMFinjectiveHoms}, F.~Enescu and M.~Hashimoto independently proved a variant of (A1) for morphisms with geometrically $F$-rational fibers (instead of regular fibers).  However, without the \emph{geometric} hypothesis on the fiber, it is unknown whether (A1) holds for $F$-rational singularities.  Even worse, in \cite[Section 4]{EnescuLocalCohomologyAndFStability}, F.~Enescu showed that $F$-injectivity can fail to satisfy (A1), more about this will be discussed in Section \ref{sec.WeakNormalityAndFailureOfBertini} below.  Additional discussion of base change problems for $F$-singularities can be found in \cite{AberbachExtensionOfWeaklyAndStronglyFregularRingsByFlat,BravoSmithBehaviorOfTestIdealsUnderSmooth,AberbachEnescuTestIdealsAndBaseChange}.


The following theorem of Hochster and Huneke will be crucial to our proof of (A1P) for sharp $F$-purity.  We originally learned this material from the exposition given in
\cite[Pages 167--175]{HochsterFoundations}.
\begin{theorem}[(Lemma 7.10 in \cite{HochsterHunekeSmoothBaseChange})]
\label{injective hull under flat local base change}
Let $(R,\bm,K)\to (S,\bn,L)$ be an arbitrary  flat local morphism.
\begin{enumerate}
\item If $z \in \bn$ is not a zero-divisor on $S/\bm S$, then $z$ is not a eradicator on $S$ and $R\to S/zS$ is again a (faithfully) flat local homomorphism.
\item More generally, if $y_1,\dots,y_t\in \bn$ form a regular sequence on $S/\bm S$, then they form a regular sequence in $S$ and $R\to S/(y_1,\dots,y_t)$ is again a faithfully flat local homomorphism. The elements $y_1,\dots,y_t$ also form an $S$-sequence on $S\otimes_RM$ for every nonzero finitely generated $R$-module M.
\item Suppose that $M$ has finite length over $R$ with $V=\Ann_M(\bm)$, and that $S/\bm S$ is zero-dimensional with socle $Q$. Then we have an injection $Q\otimes_K V\to S/\bm S\otimes_K V\cong S\otimes_RV\to S\otimes_RM$, under which $Q\otimes_KV$ is sent onto the socle in $S\otimes_RM$.
\item If  the closed fiber $S/\bm S$ is Gorenstein, $y_1,\dots,y_t\in \bn$ are elements in $S$ whose images in $S/\bm S$ form a system of parameters, $I_a=(y^a_1,\dots,y^a_t)$, then $E_S(L)\cong (\varinjlim_a S/I_a)\otimes_RE_R(K)$, where the map $S/I_a\to S/I_{a+1}$ is induced by multiplication by $y_1\cdots y_t$.  In other words, if $I=(y_1,\dots,y_t)S$ then
\[E_S(L)\cong E_R(K)\otimes_RH^t_I(S).\]
\end{enumerate}
\end{theorem}

Recall the following definition of canonical modules (when the ring is not necessarily Cohen-Macaulay) from \cite{HochsterHunekeIndecomposable}.
\begin{definition}
Let $(R,\bm,K)$ be a $d$-dimensional equidimensional Noetherian local ring. A finitely generated $R$-module $\omega$ is called a canonical module of $R$ if  it is isomorphic to $H^d_{\bm}(R)^{\vee}$.

More generally, when $R$ is not necessarily local, an $R$-module $\omega$ is called a canonical module of $R$ if $\omega_P$ is a canonical module of $R_P$ for each prime ideal $P$ of $R$.
\end{definition}

The following theorem on the behavior of canonical modules under flat extensions is certainly known to experts, but we do not know of a proof in the generality we need, \cf \cite[Theorem 3.3.14(a)]{BrunsHerzog}, \cite[Chapter V, Section 9]{Hartshorne}.

\begin{lemma}
\label{Lemma: pull back canonical module}
Let $\varphi:R\to S$ be a flat ring homomorphism with Gorenstein fibers between Noetherian equidimensional rings. Assume that $R$ admits a canonical module $\omega_R$, then $\omega_R\otimes_RS$ is also a canonical module of $S$.
\end{lemma}
\begin{proof}
It suffices to show that for each prime ideal $\bq$ of $S$, the module $\omega_{\bp}\otimes_{R_{\bp}}S_{\bq}$ is a canonical module of $S_{\bq}$ where $\bp=\bq\cap R$. Without loss of generality, we may assume that $\varphi:(R,\bm,K)\to (S,\bn,L)$ is a flat local homomorphism with Gorenstein fibers and $\omega_R$ is a canonical module of $R$. Since a module is a canonical module if and only if its completion is a canonical module of the completion of the ring, we may assume that $R$ and $S$ are complete. Then it suffices to show that $\Hom_S(\omega_R\otimes_RS,E_S)$ is isomorphic to the top local cohomology module of $S$. Choose a system of parameters $x_1,\dots,x_d$ of $R$ (assuming that $\dim(R)=d$) and elements $\underline{\bf{y}} = y_1,\dots,y_t$ in $S$ whose images in $S/\bm S$ form a system of parameters (assuming that $\dim(S/\bm S)=t$), then we have
\begin{align}
\Hom_S(\omega_R\otimes_RS,E_S)&\cong \Hom_R(\omega_R,\Hom_S(S,E_S))\notag\\
&=\Hom_R(\omega_R,E_S)\notag\\
&\cong \Hom_R(\omega_R,E_R\otimes_RH^t_{\langle\underline{{\bf{y}}}\rangle}(S))\notag
\end{align}
We claim that, for each finitely generated $R$-module $M$, a flat $R$-module $F$, and an $R$-module $N$, we have $\Hom_R(M,N\otimes_RF)\cong \Hom_R(M,N)\otimes_RF$, and we reason as follows (or simply cite \cite[Chapter II, Proposition 5.14]{HartshorneResidues}). It is clear that our claim is true when $F$ is a free $R$-module. Then since each flat $R$-module is a direct limit of free modules and $\Hom_R(M,-)$ commutes with direct limits, our claim follows. By Remark \ref{tensor of top local cohomology}(ii) we know that $H^t_{\langle\underline{{\bf{y}}}\rangle}(S)$ is a flat $R$-module, and hence we have
\begin{align}
\Hom_S(\omega_R\otimes_RS,E_S)&\cong \Hom_R(\omega_R,E_R\otimes_RH^t_{\langle\underline{{\bf{y}}}\rangle}(S))\notag\\
&\cong \Hom_R(\omega_R,E_R)\otimes_RH^t_{\langle\underline{{\bf{y}}}\rangle}(S)\notag\\
&\cong H^d_{\bm}(R)\otimes_RH^t_{\langle\underline{{\bf{y}}}\rangle}(S)\notag\\
&\cong H^{d+t}_{\bn}(S)\notag
\end{align}
This proves that $\omega_R\otimes_RS$ is a canonical module of $S$.
\end{proof}

\begin{remark}
\label{tensor of top local cohomology}
Let $R,S,y_1,\dots,y_t$ be as in Theorem \ref{injective hull under flat local base change}(iv), let $x_1,\dots,x_d\in \bm$ be elements of $S$ and let $M$ (and $N$) be an $R$-module (an $S$-module).
\begin{enumerate}
\item It is straightforward to check that
\[H^n_{(x_1,\dots,x_n)}(M)\otimes_RH^t_{(y_1,\dots,y_t)}(N)\cong H^{n+t}_{(x_1,\dots,x_n,y_1,\dots,y_t)}(M\otimes_RN),\]
and the isomorphism is given by
\[\frac{\alpha}{x^a_1\cdots x^a_n}\otimes \frac{\beta}{y^b_1\cdots y^b_t}\mapsto \frac{\alpha\otimes \beta}{x^a_1\cdots x^a_ny^b_1\cdots y^b_t}.\]
In particular, if we assume that $S/\bm S$ is regular and $y_1,\dots,y_t$ form a regular system of parameters of $S/\bm S$ and $u$ is a socle generator of $E_R(K) \cong H^{\dim R}_{\bm}(\omega_R)$, then the image of $u\otimes\frac{1}{y_1\cdots y_t}$ in $E_S(L)$, under the isomorphism
\begin{align}
E_R(K)\otimes_RH^t_I(S) &\cong H^{\dim(R)}_{\bm}(\omega_R)\otimes_RH^t_{\langle y_1, \dots, y_t \rangle}(S)\notag\\
&\cong H^{\dim(R)+t}_{\bn}(\omega_R\otimes_RS)\notag\\
& \cong H^{\dim(S)}_{\bn}(\omega_S)\ ({\rm because\ of\ Lemma\ \ref{Lemma: pull back canonical module}})\notag\\
&\cong E_S(L),\notag
\end{align}
is also a socle generator of $E_S(L)$.

\item Since each $S/(y^a_1,\dots,y^a_t)$ is flat over $R$ by Theorem \ref{injective hull under flat local base change}(ii) and $H^t_I(S)\cong \varinjlim_aS/(y^a_1,\dots,y^a_t)$, one can see that $H^t_I(S)$ is also a flat $R$-module.
\end{enumerate}
\end{remark}

\begin{lemma}
\label{lem.KeyA1PLemma}
Let $R,S$ be $F$-finite normal rings. Assume that $f:Y=\Spec(S)\to \Spec(R) = X$ is a flat morphism with regular fibers and that $\Delta$ is an effective $\bQ$-divisor on $X = \Spec R$.  If $(X,\Delta)$ is sharply $F$-pure, so is $(Y,f^*\Delta)$.
\end{lemma}

\begin{proof}
Without loss of generality, we may assume that $(R, \bm)$ and $(S, \bn)$ are local and that $f$ is a local and so faithfully flat morphism.  Additionally, by making $\Delta$ larger if needed, we may assume that $(p^e - 1)\Delta$ is integral and $(p^e-1)(K_X+\Delta)$ is Cartier and $\Hom_R(F^e_*R((p^e-1)\Delta),R)\cong F^e_*R((1-p^e)(K_X+\Delta))$ by Lemma \ref{lem.SharpFPureCanBeAssumedNice}.

We have the following composition
\[
\begin{array}{rcl}
E_R & \cong & H^d_{\bm}(R(K_R)) \\
& \to & H^d_{\bm}(F^e_* R(p^e(K_R))) \\
& \to & H^d_{\bm}(F^e_* R(p^e(K_R) + (p^e - 1)\Delta)) \\
& \cong & H^d_{\bm}(F^e_* R(K_R)) \cong F^e_* E_R
\end{array}
\]
where the last isomorphism follows since $(p^e - 1)(K_R + \Delta) \sim 0$.  Because $(R, \Delta)$ is sharply $F$-pure, this composition is injective by Lemma \ref{lem.MatlisDualSharplyFPure}.  By assumption $(S/\bm, \bn/\bm)$ is regular and local and so there exists a regular sequence ${\bf{y}} = y_1, \dots, y_t \in S$ whose images in $S/\bm$ form a regular system of parameters.

Now, we also have the following diagram:
{
\thinmuskip=0.35mu
\medmuskip=0.35mu
\thickmuskip=0.35mu
\[
{\scriptsize
\xymatrix@C=8pt{
H^d_{\bm}(R(K_R)) \ar[d] \ar[r] & H^d_{\bm}(F^e_* R(p^e K_R)) \ar[r] \ar[d] & H^d_{\bm}(F^e_* R(p^e K_R + (p^e - 1)\Delta)) \ar[d] \\
H^d_{\bm}(R(K_R)) \tensor_R H^t_{\langle{\bf{y}} \rangle}(S) \ar[r] \ar[d]_{\sim} & H^d_{\bm}(F^e_* R(p^e K_R ))\tensor_{F^e_*R}  H^t_{\langle{\bf{y}} \rangle}(F^e_* S) \ar[r] \ar[d]_{\sim} & H^d_{\bm}(F^e_* R(p^e K_R + (p^e - 1)\Delta))\tensor_{F^e_* R}  H^t_{\langle{\bf{y}} \rangle}(F^e_* S) \ar[d]_{\sim}\\
H^{d+t}_{\bm S + \langle{\bf{y}} \rangle}(S(K_S)) \ar[r] & H^{d+t}_{\bm S + \langle{\bf{y}} \rangle}(F^e_* S(p^e K_S)) \ar[r] & H^{d+t}_{\bm S + \langle{\bf{y}} \rangle}(F^e_* S(p^e K_S + (p^e -1)f^* \Delta))
}
}
\]
}
where the first vertical map is induced by sending $x \mapsto x \tensor {1 \over y_1 \dots y_t}$ and the later maps are induced by sending $x \mapsto x \tensor {1 \over y_1^{p^e} \dots y_t^{p^e}}$.  The second row of vertical maps are isomorphisms by Remark \ref{tensor of top local cohomology} and Lemma \ref{Lemma: pull back canonical module}.  Now choose a socle generator $z \in H^d_{\bm}(R(K_R)) \cong E_R$.  The vertical composition sends it to a socle generator for $H^{d+t}_{\bm S + \langle{\bf{y}} \rangle}(S(K_S)) \cong E_S$ by Remark \ref{tensor of top local cohomology}.  In particular, it follows that the left vertical map is injective.  We claim that the right vertical map is also injective.  But this follows since $H^d_{\bm}(F^e_* R(p^e K_R + (p^e - 1)\Delta)) \cong F^e_* E_R$ and one can compose the right vertical map with a multiplication by ${y_1}^{p^e - 1} \dots {y_t}^{p^e - 1}$ to obtain the Frobenius pushforward of the left vertical map.

Note that the top row $H^d_{\bm}(R(K_R)) \to H^d_{\bm}(F^e_* R(p^e K_R  + (p^e - 1)\Delta))$ is injective by hypothesis, and thus so is the composition $H^d_{\bm}(R(K_R)) \to H^{d+t}_{\bm S + \langle{\bf{y}} \rangle}(F^e_* S(p^e K_S + (p^e -1)f^* \Delta))$.  It follows that the map
\begin{equation}
\label{eq.InducedMapBetweenInjectiveHulls}
H^{d+t}_{\bm S + \langle{\bf{y}} \rangle}(S(K_S)) \to H^{d+t}_{\bm S + \langle{\bf{y}} \rangle}(F^e_* S(p^e K_S + (p^e -1)f^* \Delta))
\end{equation}
sends the (unique) socle generator to a non-zero element.  Therefore the map in Equation \ref{eq.InducedMapBetweenInjectiveHulls} is injective which proves that $(S, f^* \Delta)$ is sharply $F$-pure by Lemma \ref{lem.MatlisDualSharplyFPure}.
\end{proof}

\begin{corollary}
\label{cor.A1PForFRegFPure}
Strongly $F$-regular and sharply $F$-pure pairs satisfy $\textnormal{(A1P)}$.
\end{corollary}
\begin{proof}
By Lemma \ref{lem.KeyA1PLemma}, we already have shown that sharply $F$-pure pairs satisfy (A1P).  Choose $\Gamma$ a Cartier divisor on $\Spec R$ containing the support of $\Delta$ and such that $(\Spec R) \setminus \Supp(\Gamma)$ is regular.  We may choose $\varepsilon > 0$ such that $(R, \Delta + \varepsilon \Gamma)$ is sharply $F$-pure and thus $(S, f^* \Delta + \varepsilon f^* \Gamma)$ is also sharply $F$-pure.  Now, since the fibers of $R \to S$ are regular, $\Spec S \setminus f^* \Gamma$ is also regular by \cite[Theorem 23.7]{MatsumuraCommutativeRingTheory}.  Thus $(S, f^* \Delta)$ is strongly $F$-regular by Lemma \ref{lem.FPurePlusEpsilonImpliesFReg}.
\end{proof}

\begin{remark}
It would be natural to try to generalize the proof of Lemma \ref{lem.KeyA1PLemma} to the context of $F$-rational or $F$-injective singularities.  In particular, it is natural to try to use local cohomology modules of the rings instead of injective hulls of residue fields.  The problem is that the socle need not to be 1-dimensional.
\end{remark}

We also give a proof that (A1) holds for $F$-finite $F$-pure rings.  This was known to experts and indeed follows directly from the same argument that was used to prove \cite[Theorem 7.3]{HochsterHunekeSmoothBaseChange}.  We include the following argument for completeness however.

\begin{proposition} [(Theorem 7.3 in \cite{HochsterHunekeSmoothBaseChange})]
\label{prop.A1ForFPureNonPairs}
Let $R$ and $S$ be rings.  Assume that $f : \Spec(S) \to \Spec(R)$ is a flat morphism with regular fibers and that $R$ is $F$-pure.  Then $S$ is $F$-pure as well.  In particular, (A1) holds for $\sP = $``$F$-purity''.
\end{proposition}
\begin{proof}
The proof is similar to the one above.  In particular we may assume that $(R, \bm)$ and $(S, \bn)$ are local and that $f$ is a local morphism.  By assumption $(S/\bm, \bn/\bm)$ is regular and local and so there exists a regular sequence ${\bf{y}} = y_1, \dots, y_t \in S$ whose images in $S/\bm$ form a regular system of parameters.  Since $R$ is $F$-pure, there exists an injective map $E_R \to F^e_* E_R$ (in the $F$-finite case, this is the dual to a Frobenius splitting, in the general case see \cite[Theorem 3.2]{SharpAnExcellentFPureRingHasABigTCElement}).  Consider now the following diagram
\[
\xymatrix{
E_R \ar[d]  \ar@{^{(}->}[r] & F^e_* E_R \ar[d] \\
E_R \tensor_R H^t_{\langle {\bf{y}} \rangle}(S) \ar[r]_-{\phi} & F^e_* E_R \tensor_{F^e_* R} H^t_{\langle {\bf{y}} \rangle}(F^e_* S)
}
\]
where the vertical arrows are $z \mapsto z \tensor {1 \over y_1 \dots y_t}$ and $z \mapsto z \tensor {1 \over y_1^{p^e} \dots y_t^{p^e}}$ respectively.  These vertical maps are injective by the same argument as above.  Since $E_R \tensor_R H^t_{\langle {\bf{y}} \rangle}(S) \cong E_S$, it is sufficient to show that the bottom horizontal map is injective.  Arguing as in the previous case, we see that if $z \in E_R$ is a socle generator of $E_R$, then its image $z' \in E_S$ also generates the $S$-socle.  But then $\phi(z') \neq 0$ from the commutative diagram.
\end{proof}

\section{\textnormal{(A2)} for $F$-singularities}

In this section we prove the (A2) property for $F$-injective and $F$-pure singularities.  We also prove the (A2P) property for sharply $F$-pure pairs.  As far as we are aware, this property has only previously been explored in the special case of Gorenstein $F$-pure singularities in \cite[Theorem 4.4]{ShimomotoZhang.OnTheLocalizationProblemForFPure}.

We begin by proving that (A2P) holds for sharply $F$-pure pairs.

\begin{proposition}
\label{prop.A2PHoldsForFPurity}
The condition \textnormal{(A2P)} holds for $\sP$ = sharp $F$-purity for pairs $(R, \Delta)$ such that there exists $e > 0$ with $(p^e - 1)\Delta$ an integral Weil divisor.
\end{proposition}
\begin{proof}
The statement is local so suppose that $Z = \Spec B$ and $Y = \Spec R$.  By inverting an element of $B$, we may also assume that $B$ is regular and that $\Hom_B(F^e_* B, B)$ is free of rank $1$ as a $B$-module.  We may also assume that $B \to R$ is flat.  Set $\eta$ to be the generic point of $Z$.

We notice that (A2) holds for normality by \cite[Corollaire 9.9.5]{EGAIV3}.  Recall that for us, if $\Delta \neq 0$, then sharp $F$-purity implies normality, see Remark \ref{rem.SharpFPurePairsAreNormal}.   In particular, if $\Delta_{\eta} \neq 0$, we may assume that since the generic fiber\footnote{The fiber over the generic point.} pair is geometrically sharply $F$-pure, it is geometrically normal as well.  It follows that if $\Delta_{\eta} \neq 0$, by shrinking $Z = \Spec B$ if necessary, we may assume all the fibers are geometrically normal.

We fix the $e > 0$ that was given to us by hypothesis.
Now, because $(R \otimes_B B_{\eta}, \Delta_{\eta})$ is geometrically sharply $F$-pure, we know that the composition
\[
 (R \otimes_B F^e_* B_{\eta}) \to F^e_* (R \otimes_B F^e_* B_{\eta}) \to F^e_*(R ((p^e - 1)\Delta) \otimes_B F^e_* B_{\eta})
\]
splits by Remark \ref{rem.BasicPropertiesOfFPurePairs}(iii).  But this composition can also be factored as
\[
\begin{array}{rl}
    & (R \otimes_B F^e_* B_{\eta}) \\
\to & F^e_*(R \otimes_B B_{\eta}) \\
\to & F^e_*(R ((p^e - 1)\Delta) \otimes_B B_{\eta}) \\
\to & F^e_*(R ((p^e - 1)\Delta) \otimes_B F^e_* B_{\eta})
\end{array}
\]
and so
\[
(R \otimes_B F^e_* B_{\eta}) \to F^e_*(R ((p^e - 1)\Delta) \otimes_B B_{\eta})
\]
also splits.  Therefore, there exists a surjective map $\psi : F^e_* (R \otimes_B B_{\eta}) \to (R \otimes_B F^e_* B_{\eta})$ which factors through $F^e_*(R ((p^e - 1)\Delta) \otimes_B B_{\eta})$.

It immediately follows that there exists an element $u \in B \setminus \{ 0 \}$ such that $\psi$ is the base change with $\tensor_{B[u^{-1}]} B_{\eta}$ of a map $\psi'$
\[
F^e_* (R \otimes_B B[u^{-1}]) \to F^e_*(R ((p^e - 1)\Delta) \otimes_B B[u^{-1}]) \to R \otimes_B (F^e_* B[u^{-1}]).
\]
We may assume that $\psi'$ is also surjective.  Set $U = \Spec B[u^{-1}] \subseteq \Spec B = Z$.  For any $s \in U$ set $L \supseteq k(s)$ to be a finite extension of the residue field and tensor $\psi'$ by $\otimes_{F^e_* B} F^e_* L$ which yields:
\begin{equation}
\label{eq.RelativeFrobeniusSurj}
F^e_* (R \otimes_B L) \to F^e_*(R ((p^e - 1)\Delta) \otimes_B L ) \to R \otimes_B (F^e_* L ).
\end{equation}
which is also surjective by the right-exactness of tensor.  Choose any nonzero (and thus surjective map) $F^e_* L \to L$.  Tensoring with $R$ produces a surjective map $R \otimes_B (F^e_* L) \to R \otimes_B L$.  Finally, composing with \eqref{eq.RelativeFrobeniusSurj} yields a surjection $F^e_* (R \otimes_B L) \to R \otimes_B L$, which completes the proof.
\end{proof}

\begin{corollary}
\label{cor.A2HoldsForFPureRings}
\textnormal{(A2)} holds for $F$-finite $F$-pure rings (even in the not necessarily normal case).
\end{corollary}
\begin{proof}
This follows by exactly the same argument as in Proposition \ref{prop.A2PHoldsForFPurity}.
\end{proof}


Now we move on to (A2) for $F$-injective singularities.

\begin{proposition}
The condition \textnormal{(A2)} holds for $F$-injective singularities.
\end{proposition}
\begin{proof}
Again we assume that $Y = \Spec R$, $Z = \Spec B$ and $B \subseteq R$ is flat with $B$ a regular domain.  Consider the chain of finite maps:
\[
R \tensor_B F^e_* B \to F^e_* (R \tensor_B B) = F^e_* R \to F^e_* (R \tensor_B F^e_* B).
\]
Since $R \otimes_B B_{\eta}$ is geometrically $F$-injective (over $B_{\eta}$), the induced maps
\[
\myH^i\Big(\omega^{\mydot}_{F^e_* (R \tensor_B F^e_* B_{\eta})}\Big) \to \myH^i\Big(\omega^{\mydot}_{R \tensor_B F^e_* B_{\eta}}\Big)
\]
are surjective for each $i \in \bZ$.  It follows that the map
\[
\myH^i\Big(\omega^{\mydot}_{F^e_* R_{\eta} }\Big) \to \myH^i\Big(\omega^{\mydot}_{R \tensor_B F^e_* B_{\eta}}\Big)
\]
is also surjective for every $i$.  Choose $u \in B$ such that
\[
\myH^i\Big(\omega^{\mydot}_{F^e_* R \otimes_B B[u^{-1}] }\Big) \to \myH^i\Big(\omega^{\mydot}_{R \tensor_B F^e_* B[u^{-1}]}\Big)
\]
is also surjective.  Additionally, by using generic freeness, Theorem \ref{thm.GenFreeness},  we may assume that each $\myH^i\Big(\omega^{\mydot}_{F^e_* R \tensor_B B[u^{-1}] }\Big) \cong \myH^i\Big(\omega^{\mydot}_{F^e_* R \tensor_{F^e_* B}  F^e_* B[u^{-1}] }\Big)$ and each $\myH^i\Big(\omega^{\mydot}_{R \tensor_B F^e_* B[u^{-1}]}\Big)$  is $F^e_* B[u^{-1}]$-free (note that $F^e_* R \tensor_B B[u^{-1}] \cong F^e_* R \tensor_{F^e_* B}  F^e_* B[u^{-1}]$ since inverting $u$ gives the same module as inverting $u$ to a power).  Choose a point $s \in U = \Spec B[u^{-1}]$ and a finite extension $L \supseteq k(s)$.  Set $R_L = R \tensor_B L$.  It is now sufficient to show that $\myH^i( F^e_* \omega_{R_L}^{\mydot}) \to \myH^i(\omega_{R_L}^{\mydot})$ is surjective.  To this end, we now replace $B$ by $B_s$ and $R$ by $R \tensor_B B_s$.  We use $s$ to denote the unique maximal ideal of $B$.

Since $R$ is a finitely generated $B$-algebra, we may find a polynomial ring $A=B[y_1,\dots,y_t]$ with a surjection $A\twoheadrightarrow R = A/I$. Since $B$ is regular, so is $A$. Hence,
\[\myH^i\Big(\omega^{\mydot}_{F^e_* R} \Big) \cong \Ext^{i+\dim B + t}_A(F^e_*R,A)\ {\rm and\ }\myH^i\Big(\omega^{\mydot}_{R \tensor_B F^e_* B}\Big) \cong \Ext^{i+\dim B + t}_A(R \tensor_B F^e_* B,A). \]
Thus we know that $\Ext^i_A(F^e_*R,A)$ and $\Ext^i_A(R \tensor_B F^e_* B,A)$ are free $F^e_* B$-modules for all $i$ (and so also free $B$-modules by \cite{KunzCharacterizationsOfRegularLocalRings}).

\begin{claim}
Set $d = \dim B = \dim B_s$.  Then we have a commutative diagram:
\[
\xymatrix{
\myH^i\Big(\omega^{\mydot}_{F^e_* R }\Big) \otimes_{F^e_* B} {F^e_* L} \ar@{<->}[r]^-{\sim} \ar[d] & \myH^{i+d}\Big(\omega^{\mydot}_{F^e_* (R \otimes_{ B} {L}) }\Big) \ar[d]\\ \myH^i\Big(\omega^{\mydot}_{R \tensor_B F^e_* B}\Big) \otimes_{F^e_* B} {F^e_* L} \ar@{<->}[r]_-{\sim} & \myH^{i+d}\Big(\omega^{\mydot}_{R \tensor_{B} F^e_* L}\Big).
}
\]
where the vertical maps are the Grothendieck trace maps.
\end{claim}

\begin{proof}[of claim]
We first want to establish the Claim when $L=k(s)$. Since $B = B_s$ is local, we see that $s$ is generated by a regular system of parameters. By induction on $d=\dim(B)$ (which is the number of generators of $s$), we may assume that $s$ is generated by $x$. Since $R$ is flat over $B$, the short exact sequence $0\to F^e_*B\xrightarrow{\cdot x}F^e_*B\to F^e_*k(s)\to 0$ induces a short exact sequence $0\to R\tensor_BF^e_*B \xrightarrow{\cdot x}R\tensor_B F^e_*B\to R\tensor_BF^e_*k(s)\to 0$ (here the multiplication by $x$ is actually by the corresponding element of $F^e_* B \cong B$, or if we identify $F^e_* B$ with $B^{1/p^e}$ this is multiplication by $x^{1/p^e}$). We have the following commutative diagram
\[
\xymatrix{
0 \ar[r] & F^e_*R \ar[r]^{\cdot x} & F^e_*R \ar[r] & F^e_*(R/\langle x \rangle) \ar[r] & 0\\
0 \ar[r ]& R\tensor_B F^e_*B \ar[r]^{\cdot x} \ar[u]^{\alpha} & R\tensor_B F^e_*B \ar[r] \ar[u]^{\alpha} & R\tensor_BF^e_*k(s) \ar[r] \ar[u]^{\beta}& 0
}
\]
where $\alpha: R\tensor_BF^e_*B\to F^e_*R$ is given by $r\tensor b\mapsto br^{p^e}$ (the Radu-Andr\'{e} map) and
\[
\beta:R\tensor_B F^e_*k(s)=R\tensor_BF^e_*(B/\langle x \rangle)\to F^e_*(R/\langle x \rangle)
 \]
 is defined as $r\tensor (b+(x))\mapsto br^{p^e}+(x)$. Applying $\Hom_A(-,A)$ to this diagram, we have the following commutative diagram
\[
\xymatrix{
0\ar[r]& \frac{\Ext^i_A(F^e_*R,A)}{x\Ext^i_A(F^e_*R,A)} \ar[r] \ar[d] & \Ext^{i+1}_A(F^e_*(R/\langle x \rangle),A) \ar[r] \ar[d] & 0\\
0 \ar[r] & \frac{\Ext^i_A(R\tensor_BF^e_*B,A)}{x\Ext^i_A(R\tensor_BF^e_*B,A)} \ar[r] & \Ext^{i+1}_A(R\tensor_BF^e_*k(s),A)\ar[r] & 0
}
\]
where in the above diagram, we have 0 at the right end in both rows since both $ \Ext^{i+1}_A(F^e_*(R/\langle x \rangle),A) $ and $\Ext^{i+1}_A(R\tensor_BF^e_*k(s),A)$ are free $B$-module and hence $x$ is regular on both modules. We have proved that $\omega^{\mydot}_{R {\otimes}_{B} {k(s)}}$ is quasi-isomorphic to $(\omega^{\mydot}_R) \stackrel{\bf{L}}{\otimes_B} k(s)[d]$.  In particular, we have proved the statement for $L = k(s)$.

For each finite extension $k(s)\subset L$, tensor the above diagram with $\tensor_{A} L[y_1, \dots, y_t]$ (which is flat over $k(s)[y_1, \dots, y_t]$).  The Claim follows for $L$.
\end{proof}

Now we return to the main proof.  It follows from the right exactness of tensor that
\[
\myH^i\Big(\omega^{\mydot}_{F^e_* (R \otimes_B L) }\Big) \to \myH^i\Big(\omega^{\mydot}_{R \tensor_B F^e_* L}\Big)
\]
is surjective for all $i$.  But
\[
\myH^i\Big(\omega^{\mydot}_{R \tensor_B F^e_* L}\Big) \to \myH^i\Big(\omega^{\mydot}_{R \tensor_B L}\Big)
\]
is clearly surjective (in fact split since $R \tensor_B L \to R \tensor_B F^e_* L$ is split).  Composing these surjective maps
proves the proposition for $F$-injectivity.

\end{proof}

\section{Statement of our main theorem and corollaries}

In this section we state our main results:

\begin{theorem}
\label{thm.MainTheorem}
Suppose that $X$ is a variety over an algebraically closed field $k$, let $\Delta \geq 0$ be a $\bQ$-divisor on $X$, and let $\phi:X\to \bP^n_k$ be a $k$-morphism with separably generated (not necessarily algebraic) residue field extensions. Suppose either
\begin{itemize}
\item[(i)]  $(X,\Delta)$ is strongly $F$-regular, or
\item[(ii)]  $(X,\Delta)$ is sharply $F$-pure, or
\item[(iii)]  $\Delta = 0$, $X$ is not necessarily normal, and $X$ is $F$-pure.
\end{itemize}
Then there exists a nonempty open subset $U$ of $(\bP^n_k)^{\vee}$ such that for each hyperplane $H\in U$,
\begin{itemize}
\item[(i)]  $\left(\phi^{-1}(H),\Delta|_{\phi^{-1}(H)}\right)$ is strongly $F$-regular.
\item[(ii)]  $\left(\phi^{-1}(H),\Delta|_{\phi^{-1}(H)}\right)$ is sharply $F$-pure.
\item[(iii)]  $\phi^{-1}(H)$ is $F$-pure.
\end{itemize}\end{theorem}
\begin{proof}
In the case that $\Delta = 0$, the result for $F$-purity follows immediately from Proposition \ref{prop.A1ForFPureNonPairs}, Corollary \ref{cor.A2HoldsForFPureRings} and \cite{CuminoGrecoManaresiAxiomaticBertini} in the form of Theorem \ref{thm.1FromCGM} (again, note we never have to leave the $F$-finite setting).
We now handle the case for sharp $F$-purity.  It is clear that it is harmless to reduce to the case where $X$ is affine and then use Lemma \ref{lem.GlobalSharpFPureCanBeAssumedNice} to replace $\Delta$ by $\Delta' \geq \Delta$ such that $(p^e - 1) \Delta'$ is an integral Weil divisor.  Then the result for sharp $F$-purity is an immediate corollary of \cite{CuminoGrecoManaresiAxiomaticBertini} in the form of Theorem \ref{thm.CGMForPairs}, combined with Corollary \ref{cor.A1PForFRegFPure} and Proposition \ref{prop.A2PHoldsForFPurity}.

Now we move on to strong $F$-regularity.  Fix $D \geq 0$ a Weil divisor on $X$ such that $X \setminus D$ is non-singular and such that $\Supp(D) \supseteq \Supp(\Delta)$.  Fix $\varepsilon > 0$ such that $(X, \Delta + \varepsilon D)$ is strongly $F$-regular and so in particular sharply $F$-pure.  It follows that there exists an open subset $U \subseteq (\bP_k^n)^{\vee}$ such that $\left(\phi^{-1}(H),(\Delta + \varepsilon D)|_{\phi^{-1}(H)}\right)$ is sharply $F$-pure for all $H \in U$.  But then it follows that $\left(\phi^{-1}(H),\Delta|_{\phi^{-1}(H)}\right)$ is strongly $F$-regular by Lemma \ref{lem.FPurePlusEpsilonImpliesFReg}.
\end{proof}

\begin{remark}
\label{rem.ExtensionToTriples}
Indeed, it is easy to deduce the above result also for triples $(X, \Delta, \ba^t)$.  Let us briefly explain how:  by working on sufficiently small affine charts, the fact that $(X, \Delta, \ba^t)$ is sharply $F$-pure implies that $(X, \Delta + {1 \over p^e - 1} D)$ is also sharply $F$-pure for some Cartier divisor $D$ corresponding to a section of $\ba^{\lceil t(p^e - 1) \rceil}$ by \cite{SchwedeBetterFPureFRegular} (critically using the \emph{sufficiently small} affine charts).  Theorem \ref{thm.MainTheorem} then implies that the pair
\[
\left(\phi^{-1}(H), \Delta|_{\phi^{-1}(H)} + \left(\frac{1}{p^e - 1}\right) D|_{\phi^{-1}(H)}\right)
 \]
 is sharply $F$-pure.  But $\ba^{\lceil t(p^e - 1) \rceil}|_{\phi^{-1}(H)}$ certainly contains the section corresponding to $D|_{\phi^{-1}(H)}$ and so the proof is complete.  Using the same argument as in Theorem \ref{thm.MainTheorem}, one can also obtain the result for strongly $F$-regular triples.
\end{remark}

Since it is easy to see that sharp $F$-purity and strongly $F$-regularity are themselves open conditions (\ie, satisfy (A3P)), at least in the $F$-finite case, we obtain:

\begin{corollary}[Second Theorem of Bertini, Corollary 1 in \cite{CuminoGrecoManaresiAxiomaticBertini}]
\label{cor.BertiniForPairs}
Let $V$ be an algebraic variety over $k=\bar{k}$, and let $S$ be a finite dimensional linear system on $V$.  Further assume that $\Delta \geq 0$ is a $\bQ$-divisor\footnote{If $\Delta \neq 0$, then we assume that $V$ is normal.} on $V$. Assume that the rational map $V\dashrightarrow \mathbb{P}^n$ corresponding to $S$ induces (whenever defined) separably generated field extensions. Then for the general element $H$ of $S$, $(H, \Delta|_H)$ is also sharply $F$-pure (respectively, $(H, \Delta|_H)$ is strongly $F$-regular, respectively $\Delta = 0$ and $H$ is $F$-pure), except perhaps at the base points of $S$ and at the points of $(V, \Delta)$ which are not sharply $F$-pure (respectively, strongly $F$-regular, respectively $F$-pure).
\end{corollary}
\begin{proof}
This follows immediately from Theorem \ref{thm.MainTheorem} just as in \cite{CuminoGrecoManaresiAxiomaticBertini}.
\end{proof}

It also follows that the sharply $F$-pure locus is preserved by general hyperplane sections.

\begin{corollary}[(Corollary 2 in \cite{CuminoGrecoManaresiAxiomaticBertini})]
Let $V\subseteq \mathbb{P}^n_k$ be a normal subscheme where $k = \overline{k}$, let $\Delta \geq 0$ a $\bQ$-divisor on $V$.  Then if $(V, \Delta)$ is sharply $F$-pure (respectively, strongly $F$-regular, respectively $F$-pure and $\Delta = 0$), for the general hyperplane section $H$ of $V$, we have $(H, \Delta|_H)$ is sharply $F$-pure (respectively, strongly $F$-regular, respectively $F$-pure). Moreover, even if not, the sharply $F$-pure-locus (respectively, the strongly $F$-regular locus, respectively the $F$-pure locus) of $(V, \Delta)$ is preserved by the general hyperplane section.  For example, in the sharply $F$-pure case, this means that:
\[
\big(\textnormal{The sharply $F$-pure locus of $(X, \Delta)$}\big) \cap H \subseteq \big(\textnormal{The sharply $F$-pure locus of $(H, \Delta|_H)$}\big).
\]
\end{corollary}

\begin{remark}
By $F$-inversion of adjunction, \cite[Theorem 4.9]{HaraWatanabeFRegFPure} or \cite[Main Theorem, \cf Proposition 7.2]{SchwedeFAdjunction}, in the case that $K_V + \Delta$ is $\bQ$-Cartier with index not divisible by $p > 0$, one sees that one has ``='' instead of ``$\subseteq$'' in the above theorem.
\end{remark}

We formulate several somewhat less technical corollaries.

\begin{corollary}
\label{cor.EasyBertiniStatementPairs}
Suppose that $X$ is a normal projective variety over an algebraically closed field $k$ and that $\Delta \geq 0$ is a $\bQ$-divisor on $X$.  Suppose that $H$ corresponds to a general global section of a very ample line bundle on $X$.  Then if $(X, \Delta)$ has sharply $F$-pure (respectively strongly $F$-regular) singularities, then so does $(H, \Delta|_H)$.
\end{corollary}

\begin{corollary}
\label{cor.EasyBertiniStatement}
Suppose that $X$ is a projective variety over an algebraically closed field $k$.  Suppose that $H$ corresponds to a general global section of a very ample line bundle on $X$.  Then if $X$ has $F$-pure (respectively strongly $F$-regular) singularities, then so does $H$.
\end{corollary}

In particular, we point out that our main result has something to say in the case of Frobenius split varieties.

\begin{corollary}
\label{cor.EasyBertiniForSplit}
Suppose that $X$ is a projective Frobenius split variety over an algebraically closed field $k$.  Suppose that $H$ corresponds to a general section of a very ample line bundle on $X$.  Then $H$ is locally Frobenius split (in other words, $F$-pure).
\end{corollary}

\begin{remark}
It need not be that $H$ is globally Frobenius split in the above Corollary.  For example, consider $X = \Proj k[x,y,z,w]/\langle x^4 + y^4 + z^4 + w^4\rangle \subseteq \mathbb{P}^3$ in characteristic $p = 1 \mod 4$.  It is easy to see that this is Frobenius split but no general hyperplane section can be Frobenius split since a general hyperplane section has positive Kodaira dimension. To see this last claim, note that if $H$ is a general hyperplane section of $X$, it is a smooth degree $4$ hypersurface in $\mathbb{P}^2$ (in other words, a curve of genus $3$).
\end{remark}

We include another easy corollary which mimics common results in birational geometry, \cite[Lemma 5.17]{KollarMori}.

\begin{corollary}
\label{cor.PairsCorollary}
Suppose that $X$ is a variety over an algebraically closed field $k$, let $\Delta \geq 0$ be a $\bQ$-divisor on $X$.  Let $\phi:X\to \bP^n_k$ be a morphism with separably generated (not necessarily algebraic) residue field extensions.  Fix a general element $H$ of $(\bP^n_k)^{\vee}$.  Then:
\begin{itemize}
\item[(i)]  If $(X, \Delta)$ is sharply $F$-pure then $(X, \Delta + \phi^{-1}(H))$ is also sharply $F$-pure.
\item[(ii)]  If $(X, \Delta)$ is strongly $F$-regular, then $(X, \Delta+\phi^{-1}(H))$ is divisorially $F$-regular\footnote{The term ``divisorially $F$-regular'' unfortunately corresponds to purely log terminal singularities \cite{TakagiPLTAdjoint}.} in the sense of \cite{HaraWatanabeFRegFPure}.
\item[(iii)]  If $(X, \Delta)$ is strongly $F$-regular, then $(X, \Delta+\varepsilon \phi^{-1}(H))$ is strongly $F$-regular for all $1 > \varepsilon \geq 0$.
\end{itemize}
\end{corollary}
\begin{proof}
We prove (i) first.  We can assume $X$ is affine and then use Lemma \ref{lem.SharpFPureCanBeAssumedNice} to choose $\Delta' \geq \Delta$ such that $K_X + \Delta'$ is $\bQ$-Cartier with index not divisible by $p$ and also such that $(X, \Delta')$ is sharply $F$-pure.  We know that $(\phi^{-1}(H), \Delta'|_{\phi^{-1}(H)})$ is sharply $F$-pure and so by $F$-adjunction in the form of \cite[Theorem 4.9]{HaraWatanabeFRegFPure} or \cite[Main Theorem, \cf Proposition 7.2]{SchwedeFAdjunction}, we see that $(X, \Delta' + \phi^{-1}(H))$ is sharply $F$-pure near $\phi^{-1}(H)$.  But it is also $F$-pure away from $\phi^{-1}(H)$ by assumption.  The result follows.

The proof of statement (ii) is exactly the same as (i) but we must use \cite[Theorem 4.3]{SchwedeSmithLogFanoVsGloballyFRegular} instead of Lemma \ref{lem.SharpFPureCanBeAssumedNice}.  Finally, part (iii) follows from (i) and Lemma \ref{lem.BasicPropertiesOfFpurityAndRegularity}(c).
\end{proof}

Finally, we include a consequence to cyclic covers of Corollary \ref{cor.PairsCorollary} and also one of the main results of \cite{SchwedeTuckerTestIdealFiniteMaps}.  For a brief introduction to ramified cyclic covers, see \cite[Section 2.4]{KollarMori}.

\begin{corollary}
\label{cor.CyclicCover}
Suppose that $(X, \Delta)$ is a normal projective sharply $F$-pure pair (respectively, strongly $F$-regular pair) such that $K_X + \Delta$ is $\bQ$-Cartier.  Fix $\sL$ to be an ample line bundle and suppose that $\sL^n$ is very ample, $n$ is not divisible by $p$ and $H$ corresponds to a general section of $\sL^n$.  Let $f : Y \to X$ be the induced cyclic cover of $X$ ramified along $H$.  Then $(Y, f^* \Delta)$ is also sharply $F$-pure (respectively, strongly $F$-regular).
\end{corollary}
\begin{proof}
First note that $Y$ is normal since $H$ is general and the index is not divisible by $p$.
Certainly $(X, \Delta + H)$ is sharply $F$-pure by Corollary \ref{cor.PairsCorollary}.  Thus by \cite[Theorem 6.26]{SchwedeTuckerTestIdealFiniteMaps}, $(Y, f^* \Delta + f^* H - \Ram_f)$ is sharply $F$-pure (note that trace is surjective since $f : Y \to X$ is a finite map of order $n$ with $p \notdivide{ }  n$).  Here $\Ram_f$ is the ramification divisor of $f$ (note that $f$ is tamely ramified by construction).  Set $G$ to be the integral divisor on $Y$ such that $G = {1 \over n} f^* H$.  It follows that $f^* H = nG$ and also that $\Ram_f = (n-1)G$ since $f$ is tamely ramified.

Thus $(Y, f^* \Delta + f^* H - \Ram_f) = (Y, f^* \Delta + G)$ is sharply $F$-pure which certainly implies that $(Y, f^* \Delta)$ is sharply $F$-pure as desired.  For strong $F$-regularity, note that $f : Y \to X$ is finite \'etale outside of $G$ and $H$.  In particular, $(Y, f^* \Delta + G)$ is already strongly $F$-regular outside of $G$.  Thus $(Y, f^* \Delta)$ is strongly $F$-regular by Lemma \ref{lem.FPurePlusEpsilonImpliesFReg}.
\end{proof}

\section{Weak normality and failure of Bertini's second theorem for $F$-injective singularities}
\label{sec.WeakNormalityAndFailureOfBertini}

It is natural to ask whether (A1) is satisfied for $F$-rationality or $F$-injectivity.  For $F$-injectivity, it is known to fail by \cite[Section 4]{EnescuLocalCohomologyAndFStability}.  However, this does not {\it a priori} imply that Bertini's second theorem fails for $F$-injectivity.   However, in this section we prove that Bertini's second theorem does not hold for $F$-injective singularities even over an algebraically closed field.  This is surprising because Du~Bois singularities, the characteristic zero analog of $F$-injective singularities, are easily seen to satisfy Bertini's second theorem.

\begin{definition}[(\cite{AndreottiBombieri,YanagiharaWeaklyNormal})]
Suppose that $(R, \bm)$ is a reduced local ring of characteristic $p > 0$.  We say that $(R, \bm)$ is \emph{weakly normal} if for any $z \in K(R)$, the total ring of fractions of $R$, such that $z^p \in R$, then $z \in R$ as well.  We say that $(R, \bm)$ is \textnormal{(WN1)} if it is weakly normal and if the normalization morphism $R \to R^{\textnormal{N}}$ is unramified in codimension 1.
\end{definition}

Recall that an extension of local rings $(R,\bm) \subseteq (S,\bn)$ is unramified if
\begin{itemize}
\item[(i)]  $\bm \cdot S = \bn$ \emph{and}
\item[(ii)]  $R/\bm \subseteq S/\bn$ is separable.
\end{itemize}
We will see shortly that $F$-injective singularities can fail to be (WN1).   First we recall the picture for $F$-pure singularities:

\begin{theorem} [(Lemma 4.6 in \cite{SchwedeFInjectiveAreDuBois}, \cf Proposition 5.31 in \cite{HochsterRobertsFrobeniusLocalCohomology} and \cite{GotoWatanabeTheStructureOfOneDimensionalFPureRings})]
\label{thm.WeaklyNormalVSFinjective}
$F$-finite $F$-injective rings $R$ are weakly normal.  If $R$ is 1-dimensional, then the converse also holds.
\end{theorem}

Now we prove that $F$-pure singularities are always (WN1).

\begin{theorem}
\label{Theorem: F-purity implies WN1}
If $R$ is $F$-finite and $F$-pure, then $R$ is \textnormal{(WN1)}.  Conversely, if $R$ is $1$-dimensional, $F$-finite and \textnormal{(WN1)}, then it is also $F$-pure.
\end{theorem}
\begin{proof}
We prove the first statement.
Without loss of generality, we may assume that $(R,\bm)$ is local of dimension $1$ and so $R^{\textnormal{N}}$ is semilocal of dimension $1$.  Set $(S,\bn)$ to be the localization of $R^{\textnormal{N}}$ at an arbitrary maximal ideal over $\bm$.  Since $R$ is $F$-pure and thus seminormal, the conductor is radical in both $R$ and $R^{\textnormal{N}}$.  In particular, $\bc = \bm$.  Then $\bm S = \bc S$ is also radical and thus must be the maximal ideal $\bn$.  We need to show that $R/\bm \subseteq S/\bn$ is separable.

Since $R$ is $F$-finite and $F$-pure, we may choose a surjective map $\varphi : F^e_* R \to R$ extending to a surjective map $\varphi^{\textnormal{N}} : F^e_* R^{\textnormal{N}} \to R^{\textnormal{N}}$ by \cite[Exercise 1.2.E(4)]{BrionKumarFrobeniusSplitting}.
The conductor $\bc$ of $R \subseteq R^{\textnormal{N}}$ is $\varphi$-compatible and so we have a commutative diagram:
\[
\xymatrix{
F^e_* (S/\bn) \ar[r]   & S/\bn \\
F^e_* (R^{\textnormal{N}}/\bm) \ar[r] \ar@{->}[u] & R^{\textnormal{N}}/\bm \ar@{->}[u] \\
F^e_* (R/\bm) \ar@{^{(}->}[u] \ar[r] & R/\bm \ar@{^{(}->}[u]
}
\]
where the vertical compositions are injective.  The horizontal maps are non-zero (since they are surjective).  But non-zero $p^{-e}$-linear maps cannot extend over inseparable field extensions by \cite[Example 5.1]{SchwedeTuckerTestIdealFiniteMaps}.  This completes the proof of the forward direction.

For the converse, we fix some notation.  Note that we may assume that $(R, \bm)$ is local.  Set $k = R/\bm$ and $L = R^{\textnormal N}/\bm R^{\textnormal N}$.  Note that $L$ is a product of fields since $R \to R^{\textnormal N}$ is an unramified map of 1-dimensional rings, in particular $\bm R^{\textnormal N}$ is radical.  Fix a surjective map $\psi : F^e_* k \to k$.  Since $R$ is \textnormal{(WN1)}, every field making up $L$ is a finite separable extension of $k$.  In particular, we have a unique extension of $\psi$ to $L$, $\psi' : F^e_* L \to L$ by \cite{SchwedeTuckerTestIdealFiniteMaps}.  Since $R$ is 1-dimensional, $R^{\textnormal{N}}$ is regular, and so by \cite[Lemma 1.6]{FedderFPureRat}, there exists a map $\varphi' : F^e_* R^{\textnormal{N}} \to R^{\textnormal{N}}$ such that $\varphi'$ induces $\psi'$ by modding out by $\bm R^{\textnormal{N}}$.  Note that $\varphi'$ is surjective since $\psi'$ is.

Set $C$ to be the pullback of the diagram
\[
\big\{R^{\textnormal{N}} \to (R^{\textnormal{N}}/\bm R^{\textnormal{N}}) \leftarrow k\big\},
 \]
\cite{Ferrand2003}.  The maps $\varphi'$, $\psi'$ and $\psi$ glue together to induce a surjective map $\varphi : F^e_* C \to C$, in particular $C$ is $F$-pure.  Indeed, it is easy to see that $\varphi$ is surjective since $\varphi$ induces $\psi$ by modding out by $\bm$.  By the universal property of pullback, we have a natural map $R \to C$.  Furthermore, by construction, \cf \cite{Ferrand2003}, this map is a bijection on points which is an isomorphism outside of the maximal ideal.  Furthermore, the residue field of $C$ is $k$ as well proving that $R \to C$ is weakly subintegral \cite{YanagiharaWeaklyNormal}.  Since $R$ is weakly normal, we see that $R = C$ and so $R$ is $F$-pure as desired.
\end{proof}

We now prove that $F$-injective singularities do not satisfy Bertini-type theorems.
Our method of proof is the same as that of \cite{CuminoGrecoManaresiHyperplaneSectionsOfWNVarieties}.  There the authors prove that if a weakly normal scheme satisfies Bertini's theorem (perhaps iterated several times), then it must necessarily be \textnormal{(WN1)}.
\begin{proposition}
If $X$ is a projective surface over an algebraically closed field of characteristic $p > 0$ which is $F$-injective, except possibly at isolated points, but which fails to be \textnormal{(WN1)}.  Then a general hyperplane section is not $F$-injective (even though such a hyperplane misses the non-$F$-injective points).  In particular, the second theorem of Bertini, Corollary \ref{cor.SecondTheoremOfBertini}, fails for $F$-injective singularities.
\end{proposition}
\begin{proof}
First we comment that the existence of isolated points which are not $F$-injective is relatively harmless.  General hyperplanes will miss these points.  Such examples are also much easier to construct.

Now we prove the Proposition.  Suppose a general hyperplane $H$ was $F$-injective.  Then $H$ is weakly normal by Theorem \ref{thm.WeaklyNormalVSFinjective}, since we are working over an algebraically closed field, $H$ is also $F$-pure by \cite[Theorem 1.1]{GotoWatanabeTheStructureOfOneDimensionalFPureRings}.  Hence by Theorem \ref{Theorem: F-purity implies WN1}, $H$ must be \textnormal{(WN1)}.  However, according to \cite[Theorem 1]{CuminoGrecoManaresiHyperplaneSectionsOfWNVarieties}, $H$ being \textnormal{(WN1)} will imply that $X$ is \textnormal{(WN1)}.  We have obtained our desired contradiction since $X$ is not \textnormal{(WN1)}.
\end{proof}

Such surfaces are easy to construct.  Indeed, take any weakly normal surface $X$ which is not \textnormal{(WN1)}.  Since $F$-injective is the same as weakly normal in dimension 1, $X$ is $F$-injective except possibly at isolated points.  Such surfaces were constructed in \cite[Corollary 4]{CuminoGrecoManaresiHyperplaneSectionsOfWNVarieties}.  Since $F$-injectivity satisfies (A2) and (A3), we have obtained another proof that $F$-injectivity must fail to satisfy (A1) by \cite[Theorem 1]{CuminoGrecoManaresiAxiomaticBertini}.
We finally point out that F.~Enescu's example in \cite[Section 4]{EnescuLocalCohomologyAndFStability} is also not \textnormal{(WN1)} (and fails to satisfy (A1)) but is not finite type over an algebraically closed field.  In conclusion, we obtain the following:

\begin{theorem}
\label{thm.ExistFailureOfBertiniForSurfaces}
There exists a projective surface which is $F$-injective except possibly at finitely many points whose general hyperplane section is not $F$-injective.
\end{theorem}

\begin{remark}
\label{rem.BertiniFailuresByGluing}
It would be natural to try to find a projective $F$-injective surface which is $F$-injective \emph{everywhere} but not (WN1) and so also fails to satisfy Bertini's second theorem.  Let us briefly describe why we have failed to find such an object.
Via gluing constructions, \cf \cite{Ferrand2003}, it is not difficult to construct varieties that are \emph{proper} over an algebraically closed field, are $F$-injective but not \textnormal{(WN1)}.  Explicitly, glue two curves on a surface together, one via the identity, the other by a generically inseparable map.  However, we were unable to prove that such an example is \emph{projective} over an algebraically closed field.  In particular, we were unable to construct a surface which is projective, $F$-injective, and not \textnormal{(WN1)} although we believe such a surface should exist.
\end{remark}

\begin{remark}
It would be natural to try to find a variety of dimension $\geq 3$ which is $F$-injective but such that a general hyperplane section is not $F$-injective.  We do not know how to construct such examples however.  A natural place to look would be to try to glue together two curves on a projective threefold -- one by the identity and the other by an inseparable map.  This would yield a non-S2 scheme though.
\end{remark}

\section{Further questions and remarks}

The second theorem of Bertini is not true for $F$-injective singularities as the previous section demonstrated.  However, we can ask:

\begin{question}
Does the second theorem of Bertini hold for $F$-rational singularities?  What about for normal, or better yet (WN1), $F$-injective singularities?  Does (A1) hold for $F$-rational or normal $F$-injective singularities?
\end{question}

We can also ask:

\begin{question}
How does the classical notion of $F$-purity for pairs (sometimes called \emph{weak $F$-purity}) as defined in \cite{HaraWatanabeFRegFPure} behave under restriction to general hyperplane sections?
\end{question}

It would also be natural to try to generalize the results of this paper to test ideals.

\begin{question}
\label{quest.GenericRestrictionForTestIdeals}
Suppose that $(X, \Delta)$ is a pair where $X$ is normal and projective.  If $H$ is a general member of a very ample linear system (or more generally of a linear system whose map induces separably generated residue field extensions), is it true that $\tau(X, \Delta) \tensor \O_H = \tau(H, \Delta|_H)$?
\end{question}

\begin{remark}
\label{rem.MusYos}
Note that this question was studied before in \cite[Example 4.7]{MustataYoshidaTestIdealVsMultiplierIdeals}.  There they considered a family
$\Spec k[x,y,s] \to \Spec k[s]$ and the pair $(k[x,y,s], (x^p + y^ps)^t )$.  By choosing $1 > t \geq 1/p$, one might even expect that this example contradicts our main theorem, Theorem \ref{thm.MainTheorem}.  We explain why this is not the case.

Consider the point $\eta = \langle x^p + s y^p\rangle \in \Spec k[x,y,s]$ (or $\eta = \langle x^p + s y^p, y - \lambda \rangle$ for $0 \neq \lambda \in k$).  Note that $\eta$ lies over the generic point $\gamma$ of $\Spec k[s]$.  The residue field of $k[x,y,s]_{\eta}$ contains a $p$th root of $s$ and in particular, $k(\gamma) \subseteq k(\eta)$ not separable.  Thus there is hope that Question \ref{quest.GenericRestrictionForTestIdeals} also might have a positive answer.
\end{remark}

One can of course ask the same question for adjoint type ideals \cite{TakagiPLTAdjoint,TakagiHigherDimensionalAdjoint,SchwedeFAdjunction} and for non-$F$-pure ideals \cite{FujinoSchwedeTakagiSupplements}.

Finally, based upon Remark \ref{rem.BertiniFailuresByGluing} we ask:

\begin{question}
Does there exist a projective surface over an algebraically closed field which is $F$-injective and not \textnormal{(WN1)}?
\end{question}

\bibliographystyle{skalpha}
\bibliography{CommonBib}

\end{document}